\documentclass[12pt]{amsart}
\pdfoutput=1
\usepackage{amssymb}
\usepackage{amsmath}
\usepackage{amsfonts}
\usepackage[usenames]{color}
\usepackage{graphicx}
\usepackage{array}
\usepackage{psfrag}
\usepackage{color}
\usepackage{ulem}
\usepackage{fancyhdr}

\makeatletter
 
 \@addtoreset{equation}{section}
\makeatother

\newtheorem{theorem}{Theorem}
\newtheorem{lemma}[theorem]{Lemma}

\theoremstyle{definition}

\newtheorem{observation}[theorem]{Observation}

\theoremstyle{remark}
\newtheorem{remark}[theorem]{Remark}

\numberwithin{equation}{section}

\begin{document}

\title{The complement of a nIL graph with thirteen vertices is IL.}

\date{\today}
\author{
Andrei Pavelescu and
Elena Pavelescu
}

\address{
Department of Mathematics, University of South Alabama, Mobile, AL  36688, USA.
}

\maketitle
\rhead{The complement of a nIL graph with thirteen vertices is IL.}

\begin{abstract}
We show that for any simple non-oriented graph $G$ with at least thirteen vertices either $G$ or its complement is intrinsically linked. 
\end{abstract}
\vspace{0.1in}

\section{Introduction}

Conway and Gordon \cite{CG},  and Sachs \cite{Sa} spearheaded the theory of spatial graphs by showing that every embedding of the complete graph on six vertices $K_6$ in $\mathbb{R}^3$ contains a non-trivial link.
Such graphs are called intrinsically linked (IL), since the property is intrinsic to the graph and does not depend on the particular embedding.
If not intrinsically linked,  $G$ is said to be \textit{linklessly embeddable} (nIL).
Conway and Gordon \cite{CG} also showed that $K_7$ is an \textit{intrinsically knotted} graph, a graph whose every embedding in $\mathbb{R}^3$ contains a  cycle which is a non-trivial knot.
These results have generated a significant amount of work, including the classification of intrinsically linked graphs by Robertson, Seymour and Thomas \cite{RST}: a graph is intrinsically linked if and only if it contains one of the graphs in the Petersen family of graphs as a minor. 
For a graph $G$, \textit{a minor of G} is any graph that can be obtained from $G$ by a sequence of vertex deletions, edge deletions and edge contractions.
An edge contraction means identifying its endpoints, deleting that edge,  and deleting any double edges thus created. 
In this article, all graphs are non-oriented, without loops  and without multiple edges. 

The work in this article was motivated in part by the result of Battle et al.  \cite{BHK} saying that  for a graph with nine vertices, either the graph or its complement is nonplanar, and nine is minimal with this property.
In particular, there exists a graph $G$ with 8 vertices such that both $G$ and its complement $cG$ are planar.
An example is given by the graph induced on vertices $v_1, \ldots, v_8$ in Figure \ref{K10example} (a).
Here we look at the similar question for intrinsic linkness.

We note that there are graphs with ten vertices such that both the graph and  its complement are linklessly embeddable. 
We give an example in Figure \ref{K10example}.
In this figure, vertex $v_9$ of $G$ is adjacent to all vertices $v_1,\ldots, v_8$ of $G$ (left) and vertex $v_{10}$ of $cG$ is adjacent to all vertices $v_1,\ldots, v_9$ of $cG$ (right).
In both $G$ and $cG$, the subgraph induced by $v_1, v_2, ..., v_7$ and $v_8$ is planar.
To see why this is an example, we remember the result of Sachs that a graph is planar if and only if the cone over it is linklessly embeddable \cite{Sa}.

\begin{figure}[htpb!]
\begin{center}
\begin{picture}(340, 170)
\put(0,0){\includegraphics{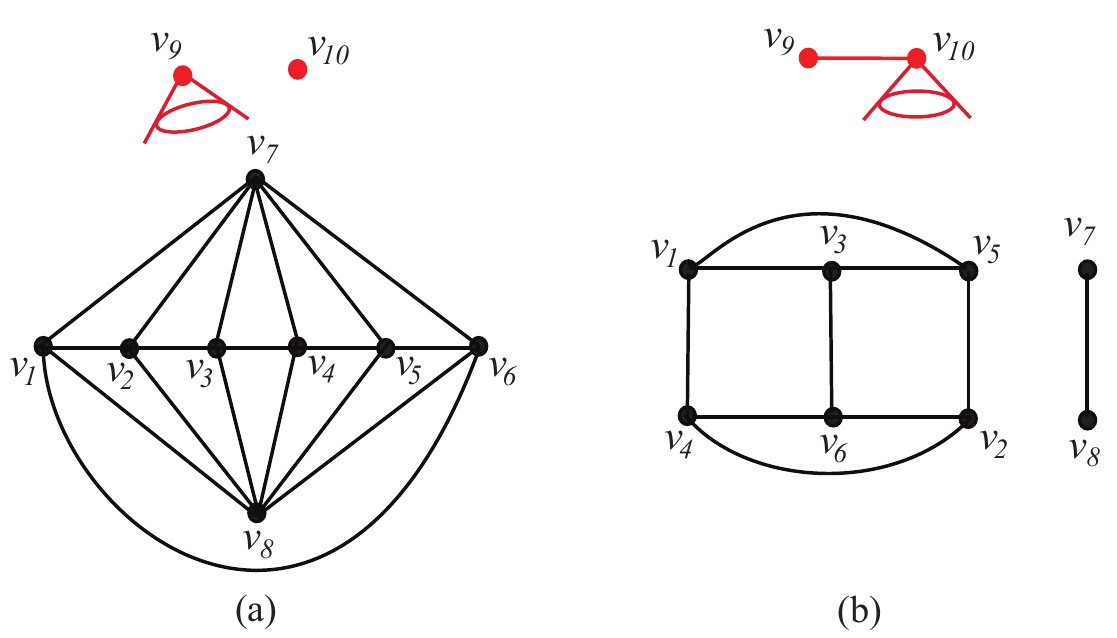}}
\end{picture}
\caption{The graphs $G$ (left) and  $cG$ (right) are both linklessly embeddable. Vertex $v_9$ of $G$ is adjacent to all vertices $v_1, \ldots, v_8$, and vertex $v_{10}$ of $cG$ is adjacent to all vertices $v_1, \ldots, v_9$.}
\label{K10example}
\end{center}
\end{figure}
Campbell et al.  \cite {CMOPRW} showed that a graph on $n\ge 6$ vertices and at least  $4n-9$ edges is intrinsically linked because it contains a $K_6$ minor.
Since the complete graph on $n$ vertices has  ${n}\choose{2}$ edges and  ${n}\choose{2}$$\ge 2(4n-9)$  for $n\ge15$, for a graph $G$ with  $n\ge 15$ vertices, either $G$ or its complement is intrinsically linked.
We improve this result by proving the following:

\begin{theorem}
If $G$ is a graph with at least thirteen vertices, either $G$ or its complement is intrinsically linked. 
\end{theorem}

\section{Notation and Background}
We first introduce notation. 
The complete graph with $n$ vertices is denoted by $K_n$.
For a graph $G$ with $n$ vertices, $cG$ represents the \textit{complement of $G$} in $K_n$.
If $\{v_1, v_2, \ldots,v_k\}$ are vertices of $G$, then $\big<v_1, v_2, \ldots,v_k\big>_G$ denotes the subgraph of $G$ induced on these vertices. 
If $v$ is a vertex of $G$, $N(v)$ is the subgraph of $G$ induced by $v$ and the vertices adjacent to $v$ in $G$. 
If $v$ is a vertex of $G$, $G-v$ is the subgraph of $G$ obtained by deleting vertex $v$ and all edges of $G$ that $v$ is incident to.
The minimum degree (maximum degree) of a vertex over all vertices of $G$ is denoted by $mindeg(G)$($maxdeg(G)$).
For a graph $G$, \textit{the cone over $G$} is the graph obtained from $G$ by adding one extra vertex and all edges between this vertex and vertices of $G$.


 Colin de Verdi\`ere introduced the new graph parameter $\mu$ based on spectral properties of matrices associated to a graph \cite{dV}.
He showed that $\mu$ is monotone under taking minors and that planarity of $G$ is equivalent to $\mu(G)\le 3$.
Lov\'asz and Schrijver \cite{LS} showed that linkless embeddability of $G$ is equivalent to $\mu(G)\le 4$.
Kotlov, Lov\'asz and Vempala \cite{KLV} showed that for a graph $G$ with $n$ vertices 
\begin{enumerate}
\item If $G$ is outerplanar, then $\mu(cG)\ge n-4$, and 
\item If $G$ is planar, then $\mu(cG)\ge n-5$. 
\end{enumerate}
These results give the following:

\begin{lemma}
\label{planarcomplement}
Consider a graph $G$ with $n\ge 10$ vertices. 
If $G$ is planar, then its complement $cG$ is IL.
\end{lemma}
\begin{proof}
By \cite{KLV}, if $G$ is planar, then $\mu(cG)\ge n-5\ge 5$. This means $cG$ is intrinsically linked \cite{LS}.
\end{proof}

\begin{lemma}
\label{deg10}
For a graph $G$ with $n$ vertices, $n\ge 11$, if there exists a vertex of $G$ whose degree is at least 10, either $G$ or its complement is intrinsically linked. 
\end{lemma}
\begin{proof}
Let $v$ be a vertex of $G$ with $deg(v) =q \ge 10$.  
Since $q \ge 10$, $N(v)$ is a cone over a graph with minimum 10 vertices, $N(v)-v$. 
If $N(v)-v$ is not planar, then $N(v)$ is intrinsically linked \cite{Sa} and thus $G$ is intrinsically linked.
If $N(v)-v$ is planar, then the complement of $N(v)-v$ in $K_q$ is intrinsically linked by Lemma \ref{planarcomplement}, and therefore $cG$ is intrinsically linked.

\end{proof}

\begin{observation}
If $G'$ is a minor of $G$ obtained through an edge contraction, then $cG'$ is a subgraph of $cG$.
\label{comp_subgraph}
\end{observation}
This observation and Lemma \ref{deg10} imply the following:
\begin{lemma}\label{deg10minor} 
Assume $G$ is a graph with at least 12 vertices. 
If an edge contraction in $G$ creates a vertex of degree at least 10, then either $G$ or $cG$ is intrinsically linked.
\end{lemma}

\begin{lemma}\label{nono} For $n\ge 12$, if $maxdeg(G)\ge 9$ then $G$ or $cG$ is intrinsically linked.
\begin{proof} If $G$ contains a vertex of degree at least 10, Lemma \ref{deg10} implies the conclusion.
Assume that $a\in V(G)$ has degree 9 and denote by $v_1,v_2,...,v_9$ its neighbors in $G$. If there exists a neighbor of $a$, say $v_1$, connecting to at least two vertices not in $N(a)$ then, since these would not be neighbors of $a$, by contracting the edge $av_1$, we obtain a minor of $G$ with at least 11 vertices and a newly created vertex $a=v_1$ of degree at least 10.
 Lemma \ref{deg10} provides the intrinsic linkness conclusion.

So we may assume that each vertex of $N(a)$ has at most one neighbor not in $N(a)$. Let $N:=\big<v_1,v_2,...,v_9\big>_G$. Since $\big<a,N\big>$ is a cone, unless $G$ is intrinsically linked,  $N$ is  planar.
By \cite{BHK}, $cN$ is nonplanar . 
In $cG$, $a$ is adjacent to $v_{10}, v_{11},...,v_{n-1}$, and each vertex $v_1,v_2,...,v_9$  connects to at least one vertex among $v_{10}, v_{11},...,v_{n-1}$, since for $n\ge 12$, $\{v_{10} ,v_{11},...,v_{n-1}\}$ has at least 2 elements.
 By contracting all the edges $av_{10}$, $av_{11}$, ...,$av_{n-1}$ in $cG$, a cone over $cN$ is obtained.
 Since $cN$ is nonplanar, $cG$ is intrinsically linked  \cite{Sa}.\\
\end{proof}
\end{lemma}


\section{Graphs with thirteen vertices }

\begin{theorem}Let $G$ denote a graph with 13 vertices. 
Then either $G$ or $cG$ is intrinsically linked.
\end{theorem}

\begin{proof} Let $k:=maxdeg(G)$. 
By Lemma \ref{nono}, we need only consider $ k \le 8$.
If $k\le 5$, then $G$ has at most 32 edges, and $cG$ has at least 46 edges and is intrinsically linked \cite{CMOPRW}. 
We look at  $k\in \{6,7,8\}$.
We show that in each case, a sequence of edge contractions in either $G$ or $cG$ creates a vertex of degree ten.
Lemma \ref{deg10} provides the intrinsic linkness conclusion.
We note that, using symmetry, the assumption that $maxdeg(G)=k$ adds that $maxdeg(cG)=k$.\\

\noindent  \textbf{Case  $k=8$}. 
Assume $deg(v_1)=8$. 
Let $v_2, v_3, ..., v_9$ be the neighbors of $v_1$ in $G$, let $N:=N(v_1)$ and $H:=\big<v_{10}, v_{11}, v_{12}, v_{13}\big>_G$.
If any of the vertices of $N$ has more than 2 neighbors in the set $\{v_{10}, v_{11}, v_{12}, v_{13}\}$, then an one edge contraction creates a vertex of degree at least 10. \\
Assume that each of  the vertices of $N$ neighbors at most 2 vertices of $H$ in $G$. 
Assume $v_2v_{10}$ and $v_2v_{11}$ are edges in $G$.
 If any of $v_i$'s, $3\le i \le 9$, neighbors both $v_{12}$ and $v_{13}$ in $G$, then contracting edges $v_1v_2$ and $v_1v_i$ creates a degree 10 vertex.
So each of the vertices $v_2,v_3,...,v_9$ misses at least one element of $\{v_{12},v_{13}\}$ in $G$, which implies that contracting the edges $v_1v_{12}$ and $v_1v_{13}$ creates a degree 10 vertex neighboring $v_2, v_3, ..., v_{11}$ in a minor of $cG$.\\

\noindent We may then assume each of the vertices $v_2,v_3,...,v_9$ neighbor at most one of the vertices of $H$ in $G$, which means there are at most 8 edges between the vertices of $N:=N(v_1)$ and those of  $H$ in $G$. 
This implies there are at least $36-8=28$ edges between $cN$ and $cH$ in $cG$.
 Since four of these edges connect $v_1$ to the vertices of $cH$, there are at least 24 edges between $\{v_2,...,v_9\}$ and $cH$ in $cG$.
Unless this lower bound is attained, the pigeonhole principle implies that one vertex of $cH$, say $v_{10}$, neighbors at least 7 vertices among  $v_2,v_3,...,v_8$ and $v_9$.
Contracting $v_1v_{10}$ creates a degree ten vertex in a minor of $cG$.  
This implies that in $cG$ each vertex of $cH$ is adjacent to at most six vertices among  $v_2,v_3,...,v_8$ and $v_9$, thus giving a maximum of $4\cdot 6=24$ edges between $cH$ and $cN\backslash v_1$ in $cG$.
It follows that in $G$ each of $v_2,v_3,...,v_8$ and $v_9$  has exactly one neighbor in $H$ and each vertex of $H$ neighbors exactly 2 of $v_2,v_3,...,v_8$ and $v_9$, thus partitioning this set into 4 blocks of 2 elements each. 
Contracting any two edges in the set $\{v_1v_{10}, v_1v_{11}, v_1v_{12}, v_1v_{13}\}$ creates a degree 10 vertex  in $cG$. \\ 

\noindent  \textbf{Case  $k=7$}.
Assume $deg(v_1)=7$. 
Notice that this implies that the minimal degree in $cG$ is 5. 
Let $v_2, v_3, ..., v_8$ be the neighbors of $v_1$ in $G$. 
If any of these vertices has more than 3 neighbors in the set $\{v_9,v_{10}, v_{11}, v_{12}, v_{13}\}$ then a one edge contraction creates a vertex of degree at least 10.  
Without loss of generality, assume $v_2$ neighbors $v_9,v_{10}$, and $v_{11}$ in $G$. 
If another $v_i$ ($3\le i \le 8$) neighbors both $v_{12}$ and $v_{13}$ in $G$, contracting the edges $v_1v_2$ and $v_1v_i$ creates  a vertex of degree 10 vertex in a 11-vertex minor of $G$.
But this implies that contracting the edges $v_1v_{12}$ and $v_1v_{13}$ in $cG$ creates a degree ten vertex. \\

It follows that there are at most 14 edges between $N:=N(v_1)$ and $H:=\big<v_9,v_{10}, v_{11}, v_{12}, v_{13}\big>_G$ in $G$. 
This implies that there are $40-5-14=21$ edges between $cH$ and $\{v_2,v_3,...,v_8\}$ in $cG$. 
By the pigeon hole principle, there is at least one vertex in $cH$, say $v_9$, which connects with at least 5 vertices in $\{v_2,v_3,...,v_8\}$. If it connects with 6 of them, then contracting edge $v_1v_9$ creates a degree 10 vertex in a minor of $cG$. 
So we may assume $v_9$  misses $v_2$ and $v_3$ in $cG$. 
Since $k=7$ for both $G$ and $ \,cG$, $v_9$ neighbors at most one other vertex of $cH$. 

If $deg_{cG}(v_9)=6$, contracting edges $v_1v_2$ and $v_2v_9$ in $G$ creates a degree 10 vertex in a minor of $G$.
If $deg_{cG}(v_9)=7$ and $v_9$ neighbors $v_{10}$ in $cG$, there is at least one element of $\{v_2,...,v_8\}$ which misses both $v_9$ and $v_{10}$ in $cG$, since otherwise one might contract $v_1v_9$ and $v_1v_{10}$ to create a degree 10 vertex in a minor of $cG$. 
This vertex must be one of the $v_2$ or $v_3$, so we assume it is $v_2$. 
Contracting $v_1v_2$ and $v_2v_9$  creates a degree 10 vertex in a minor of $G$.
\end{proof}

\noindent  \textbf{Case  $k=6$}.
Every vertex of $G$ (and $cG$) has degree 6 and $G$ has 39 edges. 
The graph $G$ is said to be 6-regular.
Denote by $v_2, v_3,...,v_7$ the neighbors of $v_1$ in $G$. 
Let $N:=N(v_1)=\big<v_2, v_3,...,v_7\big>_G$. Let $H:=\big<v_8,v_9,...,v_{13}\big>_G$.  
If any of the vertices of $N$ neighbors more than 4 of the vertices of $H$, then a one edge contraction creates a degree 10 vertex in a minor of $G$.\

Assume there exists a neighbor of $v_1$ in $G$, say $v_2$, which neighbors 4 vertices of $H$, say $v_8,v_9,v_{10}, v_{11}$. 
If for any $3\le i\le 7$, $v_i$ neighbors both $v_{12}$ and $v_{13}$, contracting $v_1v_2$ and $v_1v_i$  produces a degree 10 vertex in a minor of $G$.
 It follows that each of $v_2,v_3,..., v_6$ and $v_7$ misses at least one of the $v_{12}$ and $v_{13}$ in $G$, hence contracting 
$v_1v_{12}$ and $v_1v_{13}$ in $cG$ produces a degree ten vertex.\\

Consequently, there are at most 18 edges joining $N$ to $H$. 
Since  the sum of the degrees in $G$ of the vertices of $H$ is 36, $H$ has at least 9 edges.
Similarly, it follows that $N$ has at least $(42-18)/2=12$ edges.
Looking in the complement, we notice that $v_1\cup cH$ has at most $6+6=12$ edges and that it constitutes $N(v_1)$ in $cG$. 
By symmetry, it follows that $|N|=12$, $|H|=9$ and $|L|=18$, where $L$ denotes the set of edges joining $N$ and $H$. 
The only distribution that respects the 6-regularity of both $G$ and its complement is one in which each element of either $N -v_1$ or $H$ is incident to exactly three edges of $L$. 
This implies the subgraph $\big<v_2, ..., v_7\big>_G$  has six edges, and each vertex $v_2, ..., v_7$ has degree 2 in $\big<v_2, ..., v_7\big>_G$.
Then  $\big<v_2, ..., v_7\big>_G$ is either  two disjoint triangles or a full 6--cycle.. 
If $\big<v_2, ..., v_7\big>_G$ is two disjoint cycles, $\big<v_2, ..., v_7\big>_{cG}$ is a $K_{3,3}$. 
Then $cG$ contains a   $K_{3,3,1}$ minor given by contracting all edges of $cG$ which have $v_1$ as their end point.
So it can assumed that $v_2v_3v_4v_5v_6v_7v_2$ is cycle in $G$. 
By symmetry, $v_8, ... , v_{12}$ and $v_{13}$ are joined by all edges in the $K_6$ they determine except the cycle  $v_8v_9v_{10} v_{11}v_{12} v_{13} v_8$. 
Notice that this fully describes the adjacency relations within the neighborhood of $v_1$ and those of the set of non-neighbors of $v_1$, and the adjacency relations are the same for all vertices in $G$ and $cG$:
for a vertex $v$, every neighbor of $v$ shares two common neighbors with $v$, and any non-neighbor of $v$ has three common neighbors with $v$. 
It follows that $G$ and $cG$ are both strongly regular graphs with parameters (13,6,2,3), thus they are isomorphic to the Paley graph on 13 vertices. 
This graph contains a $K_7$ minor and is thus intrinsically linked. 
The $K_7$ minor can be obtained by contracting the following edges of $G$, highlighted in Figure \ref{Paley_graph}: $v_2v_9, v_3v_{10}, v_4v_{11}, v_5v_{12}, v_6v_{13}, v_7v_8$.

\begin{figure}[htpb!]
\begin{center}
\begin{picture}(250, 250)
\put(0,0){\includegraphics[width=3.5in]{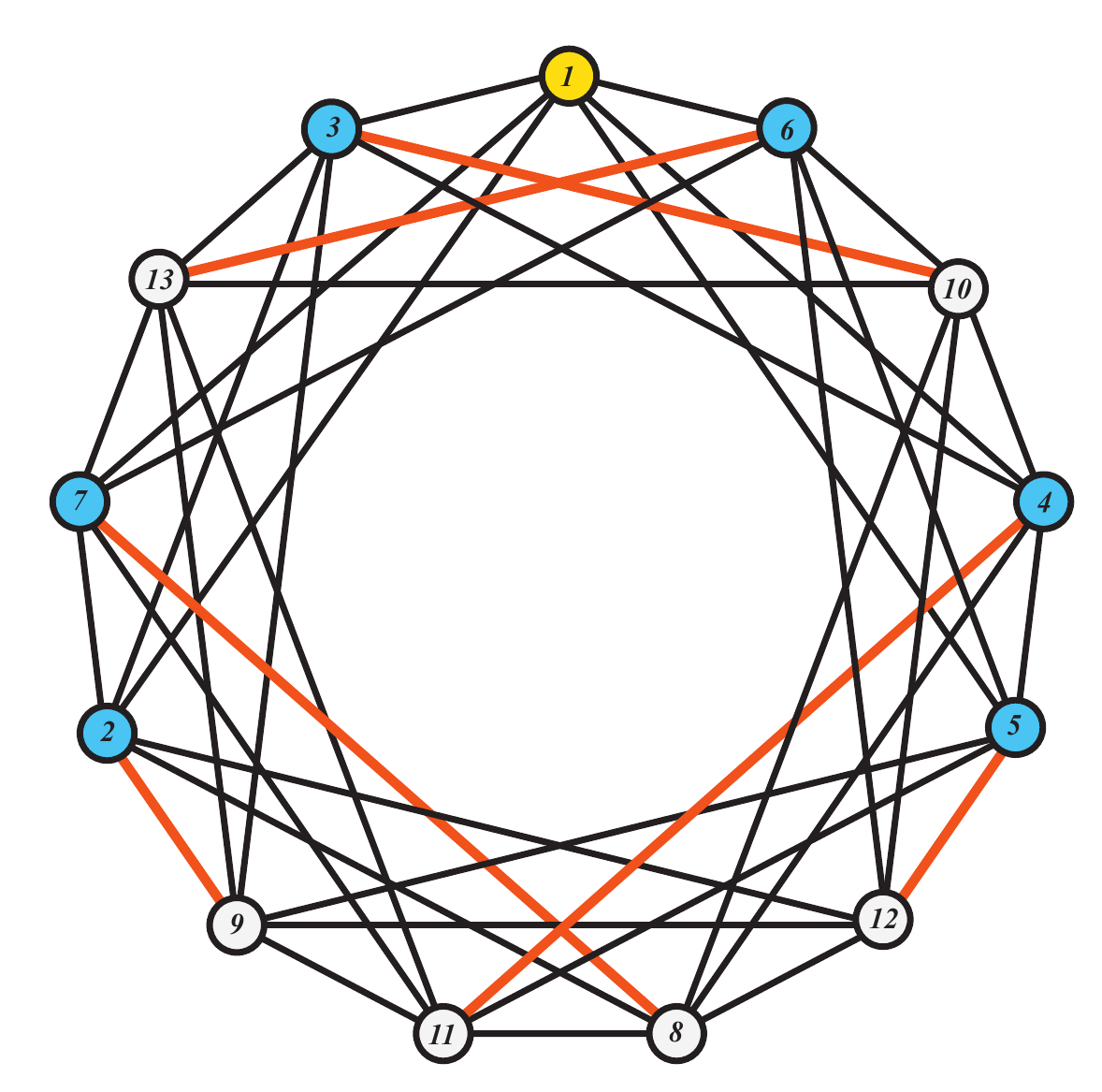}}
\end{picture}
\caption{\small  Paley graph with thirteen vertices. Contracting the highlighted edges yields a $K_7$ minor.}
\label{Paley_graph}
\end{center}
\end{figure}

\begin{remark}
It is still an open question whether the complement of a linklessly embeddable graph with 11 or 12 vertices is necessarily intrinsically linked.  
\end{remark}
%
%
%

\bibliographystyle{amsplain}

\end{document}